%% file: GMII.tex
\documentclass[a4paper]{article}
\usepackage{enumerate}
\input{macros}

\input{environs}
\newcommand{\SFS}{Seifert fibre space}

\title{Profinite rigidity of graph manifolds, II: knots and mapping classes}
\author{Gareth Wilkes}
\begin{document}
\maketitle

\input{intro}
\input{previously}
\input{examples}
\input{commensurability}
\input{knots}
\input{mcg}

\bibliographystyle{alpha}
\bibliography{graphmflds}
\end{document}

%% file: macros.tex
\usepackage{amsmath}
\usepackage{amsfonts}
\usepackage{amssymb}
\usepackage{amsthm}
\usepackage{mathtools}
\usepackage{tikz}
\usepackage{tikz-cd}
\usetikzlibrary{patterns}
\DeclareMathOperator{\id}{id}

\DeclareMathOperator{\Out}{Out}
\newcommand{\bdy}{\ensuremath{\partial}}
\newcommand{\sph}[1]{\ensuremath{\mathbb{S}^{#1}}}

\newcommand{\iso}{\ensuremath{\cong}}
\newcommand{\Z}[1][]{\ensuremath{\mathbb{Z}_{#1}}}
\newcommand{\Q}{\ensuremath{\mathbb{Q}}}
\newcommand{\R}{\ensuremath{\mathbb{R}}}
\newcommand{\I}{\ensuremath{\mathbb{I}}}

\newcommand{\HR}{\ensuremath{\mathbb{H}^2\times\R}}

\newcommand{\ofg}[1][]{\ensuremath{\pi_1^{\text{orb}}#1}}

\newcommand{\gp}[1]{\ensuremath{\langle #1\rangle}}

%% file: environs.tex
\newtheorem{theorem}{Theorem}[section]

\newtheorem{prop}[theorem]{Proposition}

\newtheorem{lem}[theorem]{Lemma}
\newtheorem{clly}[theorem]{Corollary}
\newtheorem{question}[theorem]{Question}
\theoremstyle{definition}
\newtheorem{defn}[theorem]{Definition}
\newtheorem*{cnv}{Conventions}

\theoremstyle{remark}
\newtheorem*{rmk}{Remark}
\newtheorem{example}[theorem]{Example}
\theoremstyle{plain}

\newcounter{introthmcount}
\newenvironment{introthm}[1]{\\[1.9ex] {\bf Theorem\refstepcounter{introthmcount} \label{#1}\Alph{introthmcount}.} \em}{\em \\[2ex]}

\theoremstyle{definition}

%% file: intro.tex
\begin{abstract}
In this paper we study some consequences of the author's classification of graph manifolds by their profinite fundamental groups. In particular we study commensurability, the behaviour of knots, and relation to mapping classes. We prove that the exteriors of graph knots are distinguished among all 3-manifold groups by their profinite fundamental groups. We also prove a strong conjugacy separability result for certain mapping classes of surfaces.
\end{abstract}
\section*{Introduction}
There has been a recent slew of papers dealing with the properties of 3-manifold groups which may be detected via the finite quotients of the group, or equivalently via its profinite completion. Examples of major advances in this field include the detection of fibring for compact 3-manifolds by Jaikin-Zapirain \cite{Jaikin17}; the detection of the geometry of a closed 3-manifold by Wilton and Zalesskii \cite{WZ10}; the proof that each once-punctured torus bundle is {\em profinitely rigid} \cite{BRW16} (that is, it may be distinguished from all other 3-manifolds by the profinite completion of its fundamental group); and the classification of those Seifert fibre spaces with the same profinite completions of fundamental groups by the author \cite{Wilkes15}.

In the paper \cite{Wilkes16} the author proved a result classifying graph manifolds by the profinite completions of their fundamental groups. This classification (see Theorem \ref{GMrigid} of the present paper) takes the form of a finite list of numerical conditions which are easy to check for a given pair of graph manifolds when presented in a certain standard form. This present paper represents a continuation of \cite{Wilkes16}. A certain familiarity with at least Section 10 of \cite{Wilkes16} will be required. Initially we seek to shed light on the classification theorem by means of examples. In Section \ref{secComm} we prove the following corollary (Proposition \ref{commens}).
\begin{introthm}{introCommens}
Every closed orientable graph manifold has a finite-sheeted cover with profinitely rigid fundamental group. Hence if two graph manifold groups have isomorphic profinite completions, then they are commensurable.
\end{introthm} 
Following this we will use the techniques from \cite{Wilkes16} to investigate two other entities of great interest to low dimensional topologists---knots and mapping classes. In Section \ref{secKnots} we will study those knot exteriors in \sph{3}\ which are graph manifolds and prove that they are all determined by the profinite completions of their fundamental groups. Strikingly this result does not assume any condition on the behaviour with respect to the peripheral structures of the groups of the isomorphisms of profinite completions. The following appears as Theorem \ref{thmKnots}.
\begin{introthm}{introKnots}
Let $M_K$ be the exterior of a graph knot $K$. Let $N$ be another compact orientable 3-manifold and assume that $\widehat{\pi_1 M_K}\iso \widehat{\pi_1 N}$. Then $\pi_1 M_K\iso \pi_1 N$. In particular if $K$ is prime and $N$ is also a knot exterior then $N$ is homeomorphic to $M_K$.
\end{introthm}
Finally in Section \ref{secMCG} we use the behaviour of profinite completions of fibred graph manifolds to deduce the following result (Theorem \ref{thmMCG}) concerning mapping classes. Here a `piecewise periodic' mapping class is one for which the corresponding mapping torus is a graph manifold.
\begin{introthm}{introMCG}
If $\phi_1$ and $\phi_2$ are piecewise periodic automorphisms of a closed surface group $\pi_1 S$ which are not conjugate in $\Out(\pi_1 S)$, then $\phi_1$ is not conjugate to $\phi_2^{\kappa}$ in $\Out(\widehat{\pi_1 S})$ for any $\kappa \in\hat\Z{}^{\!\times}$.
\end{introthm}
The author is grateful to Marc Lackenby for carefully reading this paper, and to Michel Boileau for helpful and insightful conversations. The author was supported by the EPSRC.  
\begin{cnv}
In this paper, we will adopt the following conventions.
\begin{itemize}
\item Generally speaking, discrete groups will generally be given Roman letters $A,G,H,...$. 
\item A finite graph of spaces will be denoted ${\cal M}= (X,M_\bullet)$ where $X$ is a finite graph and $M_\bullet$ will be an edge or vertex space. Similarly for graphs of groups.
\item For us, a graph manifold will be required to be non-geometric, i.e. not a single \SFS{} or a Sol-manifold, hence not a torus bundle. All 3-manifolds will be orientable.
\item For two elements $g,h$ of a group, $g^h$ will denote $h^{-1}gh$. That is, conjugation will be a right action.
\item We will sometimes shorten the phrase `profinite completion of the fundamental group' to `profinite fundamental group'. In such cases we may use the term `discrete fundamental group' to mean the fundamental group itself.
\end{itemize}
\end{cnv}

%% file: previously.tex
\section{Results from previous papers}
In this section we shall recall those results and notions from \cite{Wilkes16} and other papers which will be required for our discussions. We begin with some information and conventions concerning 2-orbifolds and Seifert fibre spaces.

\begin{defn}[Definition 2.6 of \cite{Wilkes16}]\label{exauto}
Let $O$ be an orientable 2-orbifold of genus $g$ with $s+1$ boundary components and $r$ cone points, with fundamental group
\[B= \ofg[O]=\left< a_1,\ldots, a_r, e_1, \ldots, e_s, u_1, v_1, \ldots, u_g,v_g\mid a_i^{p_i}\right>\]
where the boundary components of $O$ are represented by the conjugacy classes of the elements $e_1, \ldots, e_s$ together with
\[ e_0 = \left(a_1 \cdots a_r e_1\cdots e_s [u_1, v_1]\cdots [u_g,v_g]\right)^{-1} \]
Let $\mu\in\widehat\Z{}^{\!\times}$. An {\em exotic automorphism of $O$ of type $\mu$} is an automorphism $\psi\colon\widehat B\to\widehat B$ such that $\psi(a_i) \sim a_i^\mu$ and $\psi(e_i)\sim e_i^\mu$ for all $i$, where $\sim$ denotes conjugacy in $\widehat B$.
\end{defn}
There is a corresponding notion for non-orientable orbifolds, but this will not be relevant to this paper.
\begin{prop}[Proposition 10.5 of \cite{Wilkes16}]\label{orbexauto}
Let $O$ be an orientable 2-orbifold with boundary. Then $O$ admits an exotic automorphism of type $\mu$ for any $\mu\in\widehat\Z$. Moreover this automorphism may be induced by an automorphism of the orbifold $\mathring O$ obtained from $O$ by removing a small disc about each cone point.
\end{prop}

\begin{theorem}[Theorem 2.7 of \cite{Wilkes16}; adaption of Theorems 5.8 and 5.9 of \cite{Wilkes15}]\label{Wilkeswbdy}
Let $M$ and $N$ be Seifert fibre spaces whose boundary components  are $\bdy M_1,\ldots,\bdy M_n$ and $\bdy N_1,\ldots,\bdy N_n$ respectively. Suppose $\Phi$ is an isomorphism of group systems
\[ \Phi\colon (\widehat{\pi_1 M};\widehat{\pi_1\bdy M_1},\ldots, \widehat{\pi_1\bdy M_n})\to (\widehat{\pi_1 N};\widehat{\pi_1\bdy N_1},\ldots, \widehat{\pi_1\bdy N_n}) \]
(where we allow the replacement of peripheral subgroups by conjugates). Then:
\begin{itemize}
\item If $M$ is a twisted $\I$-bundle over the Klein bottle or a copy of $\sph{1}\times\sph{1}\times\I$, then $M=N$.
\item Otherwise, the base orbifolds of $M$ and $N$ may be identified with the same orbifold $O$ such that $\Phi$ splits as an isomorphism of short exact sequences
\[\begin{tikzcd}
1\ar{r} & \widehat\Z\ar{r} \ar{d}{\cdot \lambda} & \widehat{\pi_1 M} \ar{r} \ar{d}{\Phi} & \widehat{\ofg[O]} \ar{r} \ar{d}{\phi} & 1 \\
1\ar{r} & \widehat\Z\ar{r} & \widehat{\pi_1 N} \ar{r} & \widehat{\ofg[O]} \ar{r} & 1
\end{tikzcd}\]
where $\lambda$ is some invertible element of $\widehat\Z$ and $\phi$ is an exotic  automorphism of $O$ of type $\mu$.

Hence if $N$ is a surface bundle over the circle with fibre a hyperbolic surface $\Sigma$ with periodic monodromy $\psi$, then $M$ is also such a surface bundle with monodromy $\psi^k$ where $k$ is congruent to $\kappa=\mu^{-1}\lambda$ modulo the order of $\psi$.   
\end{itemize} 
\end{theorem}
In this theorem, as elsewhere, we follow the convention that the short exact sequence is the profinite completion of a short exact sequence for the discrete fundamental group, so that the generator $1\in\widehat\Z$ is the homotopy class of a regular fibre rather than an arbitrary generator of $\widehat\Z$.
\begin{defn}\label{defHempelpair}
If $M,N$ are as in the latter case of the above theorem, we say that $(M,N)$ is a {\em Hempel pair of scale factor $\kappa$}, where $\kappa= \mu^{-1}\lambda$. Note that $\kappa$ is only well-defined modulo the order of $\psi$, which may be taken to be the lowest common multiple of the orders of the cone points of $M$. Note that a Hempel pair of scale factor $\pm 1$ is a pair of homeomorphic \SFS{}s.
\end{defn}

Finally we state the results from \cite{Wilkes16} concerning graph manifolds. The reader is warned that at certain points in this paper we will be needing details from the proofs of these theorems as well as just the statements.

\begin{theorem}[Theorem 6.2 of \cite{Wilkes16}]\label{decompfixed}
Let $M$ and $N$ be closed graph manifolds with respective JSJ decompositions $(X,M_\bullet)$ and $(Y,N_\bullet)$. Assume that there is an isomorphism $\Phi\colon \widehat{\pi_1 M}\to \widehat{\pi_1 N}$. Then the isomorphism $\Phi$ induces an isomorphism of JSJ decompositions in the following sense:
\begin{itemize}
\item there is a graph isomorphism $\varphi\colon X\to Y$;
\item $\Phi$ restricts to an isomorphism from $\widehat{\pi_1 M_x}$ to a conjugate of $\widehat{\pi_1 N_{\varphi(x)}}$ for every $x\in V(X)\cup E(X)$.
\end{itemize}
After performing an automorphism of $\widehat{\pi_1 N}$, one may eliminate the conjugacy in the last point.
\end{theorem}

We will henceforth restrict attention to those graph manifolds whose vertex spaces have orientable base orbifolds. Let us recall how one obtains numerical invariants of a graph manifold $M$. Let the JSJ decomposition be $(X, M_\bullet)$. In this section we shall adopt the convention (from Serre) that each `geometric edge' of a finite graph is a pair $\{e,\bar e\}$ of oriented edges, with $\bar e$ being the `reverse' of $e$. Fix presentations in the standard form
\[\big< a_1,\ldots, a_r, e_1, \ldots, e_s, u_1, v_1, \ldots, u_g, v_g, h\,\big|\, a_i^{p_i} h^{q_i}, h \text{ central}\, \big>\]
for each $M_x\, (x\in V(X))$. This determines an ordered basis $(h, e_i)$ for the fundamental group of each boundary torus of $M_x$, where the final boundary component is described by 
\[e_0 = (a_1\cdots a_r e_1\cdots e_s [u_1, v_1]\cdots[u_g, v_g])^{-1}\]
Then with these bases the gluing map along an edge $e$ (from the boundary component of $M_{d_0(e)}$ to that of $M_{d_1(e)}$) takes the form of a matrix (acting on the left of a column vector) 
\[\begin{pmatrix}
\alpha(e) & \beta(e)\\ \gamma(e) & \delta(e)
\end{pmatrix} \] 
where $\gamma(e)$, the intersection number of the fibre of $d_0(e)$ with that of $d_1(e)$, is non-zero by the definition of a graph manifold. The number $\gamma(e)$ is well-defined up to a choice of orientation of the fibres of the two vertex groups. This matrix has determinant $-1$ from the requirement that the graph manifold be orientable. Once an orientation of the fibre and base are fixed, the number $\delta(e)$ becomes independent of the choice of presentation, modulo $\gamma(e)$. Changing these orientations multiplies the matrix by $-1$. The precise values of the $\delta(e)$ and $q_i$ may be changed by arbitrary integer multiples of $\gamma(e)$ and $p_i$ using `Dehn twists' of the form
\[ e_i\mapsto e_i h,\,e_j\mapsto e_jh^{-1} \text{ or } e_i\mapsto e_i h^{\pm1},\,a_j\mapsto a_jh^{\mp1}\]
These operations however leave the {\em total slope}
\[ \tau(x) = \sum_{d_0(e)=x} \frac{\delta(e)}{\gamma(e)}  - \sum \frac{q_i}{p_i} \]
of the vertex space invariant. Note also that these quantities are all invariant under the conjugation action of the group on itself, i.e.\ it does not matter which conjugate of each edge or vertex group we consider.

\begin{theorem}[Theorem 10.9 of \cite{Wilkes16}]\label{GMrigid}
	Let $M$ and $N$ be closed orientable graph manifolds with JSJ decompositions $(X, M_\bullet)$ and $(Y, N_\bullet)$ respectively, in which all Seifert fibred pieces have orientable base orbifold. Suppose $M$ and $N$ are not homeomorphic.
	 \begin{enumerate}
		\item If $X$ is not bipartite, then $\pi_1 M$ is profinitely rigid.
		\item If $X$ is bipartite, on two sets $R$ and $B$, then $\pi_1 M$ and $\pi_1 N$ have isomorphic profinite completions if and only if, for some choices of orientations on fibre subgroups, there is a graph isomorphism $\phi\colon X\to Y$ and some $\kappa\in\widehat{\Z}{}^{\!\times}$ such that:
	\begin{enumerate}
\item[(a)] For each edge $e$ of $X$, $\gamma(\phi(e))=\gamma(e)$
\item[(b)] The total slope of every vertex space of $M$ or $N$ vanishes
\item[(c1)] If $d_0(e)=r\in R$ then $\delta(\phi(e))=\kappa \delta(e)$ modulo $\gamma(e)$, and $(M_r, N_{\phi(r)})$ is a Hempel pair of scale factor $\kappa$.
\item[(c2)] If $d_0(e)=b\in B$ then $\delta(\phi(e))= \kappa^{-1} \delta(e)$ modulo $\gamma(e)$, so that $(M_b, N_{\phi(b)})$ is a Hempel pair of scale factor $\kappa^{-1}$. 
	\end{enumerate}
	\end{enumerate}
\end{theorem}
\begin{rmk}
In the present paper we shall only need to deal with manifolds whose Seifert fibred pieces have orientable base orbifold, so we have cleaned up this statement slightly from Theorem 10.9 of \cite{Wilkes16} which deals with the general case.
\end{rmk}

%% file: examples.tex
\section{Examples}\label{secEgs}
We now give some simple illustrative examples to demonstrate some of the phenomena that may occur in consequence of Theorem \ref{GMrigid}. For some of the examples we will also explicitly describe the isomorphisms of profinite completions.
\begin{example}[Changing a vertex space]\label{ExChVxSp}
Fix a positive integer $p$ and let $0<q<p/2$ be coprime to $p$. Consider the following family of graph manifold groups, whose vertex groups for different values of $q$ are not isomorphic. The manifolds themselves are schematically illustrated in Figure \ref{FigChVxSp}.
\begin{eqnarray*}
G_q &=& \big< a_1,a_2,u,v,h,u',v',h'\,\big|\, a_1^p h^q, a_2^p h^{-q}, e'=h, h'=e,\\ 
&&\qquad\qquad [h,a_1], [h,a_2], [h,u], [h,v], [h', u'], [h', v']\big>\\
&=& \gp{a_1,a_2,u,v,h} \ast_{\Z^2} \gp{u',v',h'}
\end{eqnarray*} 
where $e=(a_1 a_2[u,v])^{-1}$ and $e'=[u',v']^{-1}$. One may readily check that the conditions of Theorem \ref{GMrigid} are satisfied so that all these groups as $q$ varies have isomorphic profinite completions. The isomorphism from $\widehat G_1$ to $\widehat G_q$ defined in the theorem may be described as follows. Let $\kappa= q+p\rho$ be an element of $\widehat\Z {}^{\!\times}$ congruent to $q$ modulo $p$.
Let $\phi$ be an automorphism of the free profinite group on the generators $u'$ and $v'$ such that 
\[\phi([u',v']) = [u', v']^\kappa\] i.e.\ an exotic automorphism of type $\kappa$ of a once-punctured torus, which exists by Proposition \ref{orbexauto}. Then define $\Phi\colon\widehat G_1\to\widehat G_q$ by sending $u'$ and $v'$ to their images under $\phi$, by mapping
\[h\mapsto h^\kappa, \quad a_1\mapsto a_1 h^{-\rho},\quad a_2\mapsto a_2 h^\rho \] 
and by the `identity' on all other generators. The reader may readily verify that this gives a well-defined surjection of profinite groups. As argued in the proof of the theorem it is in fact an isomorphism. This may also be seen from the fact that the map so given is an isomorphism of graphs of profinite groups.
\end{example}
\begin{figure}[htp]
\centering
\includegraphics[scale=0.6]{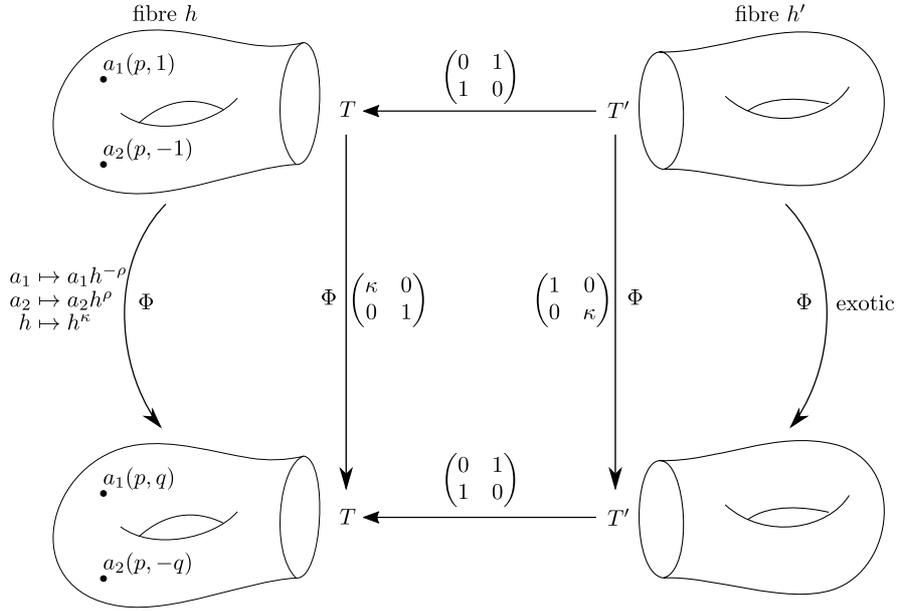}
\caption{A schematic depiction of the graph manifolds considered in Example \ref{ExChVxSp}. Here we represent each Seifert fibred piece as a surface with some marked points, where the marked points represent exceptional fibres and are labelled with the Seifert invariants of that fibre. Each boundary subgroup $T^{(\prime)}$ is given the ordered basis $(h^{(\prime)}, e^{(\prime)})$.   }
\label{FigChVxSp}
\end{figure}
\begin{example}[Changing a gluing map]\label{ExChGlMp}
Consider the two graph manifolds depicted schematically in Figure \ref{FigChGlMp}. Each is composed of two product \SFS{}s $S\times\sph{1}$ and $S'\times\sph{1}$ glued together, where $S$ and $S'$ are copies of a torus with two open discs removed. One readily verifies that the conditions of Theorem \ref{GMrigid} hold, so the fundamental groups have isomorphic profinite completions. 
\begin{figure}[htp]
\centering
\includegraphics[scale=0.7]{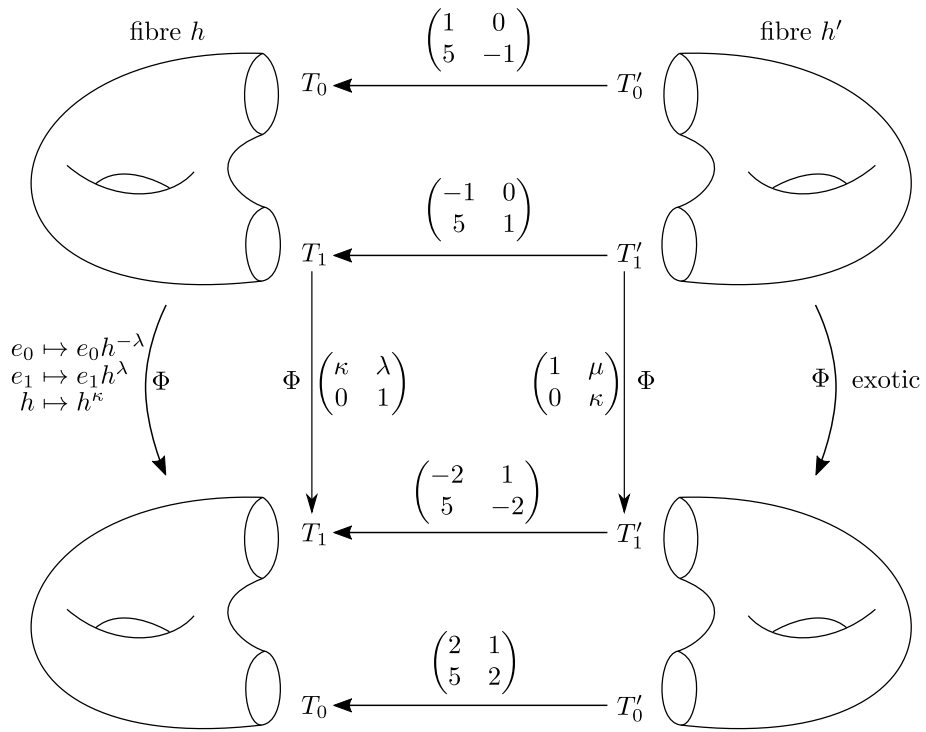}
\caption{The graph manifolds considered in Example \ref{ExChGlMp}. Here each Seifert fibred piece is product of a surface with \sph{1}. Each boundary subgroup $T_i^{(\prime)}$ is given an ordered basis $(h^{(\prime)}, e_i^{(\prime)})$.  Please note that for clarity we have swapped the roles of $T_0$ and $T_1$ in the lower diagram and omitted the maps induced by $\Phi$ on $T_0$ and $T'_0$. }
\label{FigChGlMp}
\end{figure}
As in the previous example, we will write down the isomorphism explicitly. The fundamental group of the first manifold has presentation
\begin{multline*}
G_1 = \big<e_1,u,v,h,e'_1, u',v',h', t\,\big|\, [h,e_1], [h,u], [h,v], [h',e'_1], [h', u'], [h', v']\\ 
e'_1=e_1, h'=h^{-1}e_1^5, (e'_0)^t= e_0^{-1}, (h')^t=he_0^5\big>
\end{multline*}
where $e_0= (e_1[u,v])^{-1}$ and similarly for $e'_0$. The second group has presentation
\begin{multline*}
G_2 = \big<e_1,u,v,h,e'_1, u',v',h', t\,\big|\, [h,e_1], [h,u], [h,v], [h',e'_1], [h', u'], [h', v']\\ 
e'_1=he_1^{-2}, h'=h^{-2}e_1^5, (e'_0)^t= he_0^{2}, (h')^t=h^2 e_0^5\big>
\end{multline*}
Let $\kappa\in\widehat\Z {}^{\!\times}$ be congruent to 2 modulo 5, so that $-2\kappa$ is congruent to 1 modulo 5. Take $\lambda,\mu\in \widehat\Z$ such that
\[\kappa=2+5\lambda, \quad 1=-2\kappa + 5 \mu \] Let $\phi$ be an exotic automorphism of $S'$ of type $\kappa$, such that
\[\phi(e'_1)=(e'_1)^\kappa, \quad \phi(e'_0)=[(e'_0)^\kappa]^g \]
for some $g$ in the subgroup of $\widehat{G}_2$ generated by $u'$, $v'$ and $e'_1$. Now define $\Phi\colon\widehat G_1\to\widehat G_2$ as follows:
\begin{eqnarray*}
h\mapsto h^\kappa, & e_1\mapsto e_1h^\lambda, & u\mapsto u,\, v\mapsto v,\,  t\mapsto g^{-1}t \\
h'\mapsto h', &e'_1\mapsto \phi(e'_1)(h')^\mu, & u'\mapsto \phi(u'),\, v'\mapsto \phi(v') 
\end{eqnarray*} 
The reader is left to verify that this map is well-defined. The only real issue is whether the maps on the edge tori match up correctly. As indicated in Figure \ref{FigChGlMp}, this amounts to checking a matrix equation 
\begin{eqnarray*}
\begin{pmatrix} -2 & 1 \\ 5 & -2 \end{pmatrix}
&=&\begin{pmatrix} \kappa & \lambda \\ 0 & 1 \end{pmatrix}
\begin{pmatrix} -1 & 0 \\ 5 & 1 \end{pmatrix}
\begin{pmatrix}  1 & \mu \\ 0 & \kappa \end{pmatrix}^{-1}\\
&=& \begin{pmatrix}
-\kappa+5\lambda & \mu-5\lambda\kappa^{-1}\mu+\lambda\kappa^{-1}\\
5 & \kappa^{-1}(1-5\mu)
\end{pmatrix}
\end{eqnarray*}
on the `$e_1$ edge' (and a similar one on the `$e_0$ edge'). These equations hold by the definitions of $\lambda$ and $\mu$.
\end{example}

\begin{example}[Fibred examples]\label{ExFibGM}
Consider the surface $S$ formed from a sphere by removing 10 small discs spaced equidistantly along the equator. This surface has an order 5 self-homeomorphism $\varphi$ given by a rotation. The surface bundle $M_q = S \rtimes_{\varphi^q} \sph{1}$ with monodromy $\varphi^q$ (for $q$ coprime to 5) is a Seifert fibre space whose base orbifold has genus 0, two boundary components and two exceptional fibres, each of order 5 and with Seifert invariants $(5, q)$ and $(5, -q)$ in some order. The surfaces $S$ describe parallel circles on each boundary torus; choose one such curve to give the second basis element of the fundamental group of the boundary. Take two such \SFS{}s $M_{q_1}$ and $M_{q_2}$. Glue the `$e_0$ boundary' of $M_{q_2}$ to the `$e_0$ boundary' of $M_{q_1}$ by a map \[\begin{pmatrix}-1 & 0 \\ 1 & 1\end{pmatrix} \] (interpreted as a map from the boundary of $M_{q_2}$ to the boundary of $M_{q_1}$) and glue the `$e_1$ boundary' of $M_{q_2}$ to the `$e_1$ boundary' of $M_{q_1}$ by a map \[\begin{pmatrix} -1 & 0 \\ -1 & 1\end{pmatrix} \]
These gluings give a graph manifold $L_{q_1, q_2}$ (see Figure \ref{FigFibGM}). The choice of the second column guarantees that the glued-up manifold is still fibred, since in each case the fibre surfaces, each running exactly five times over the boundary components in curves isotopic to $e_0$ or $e_1$, may be matched by this gluing homeomorphism. By construction the total slope at each vertex space $M_{q_i}$ is zero. We may now apply Theorem \ref{GMrigid} to conclude that the distinct fibred graph manifolds $L_{1,1}$ and $L_{2, -2}$ have isomorphic profinite completions of fundamental groups.
\end{example}
\begin{figure}[htp]
\centering
\includegraphics[scale=0.7]{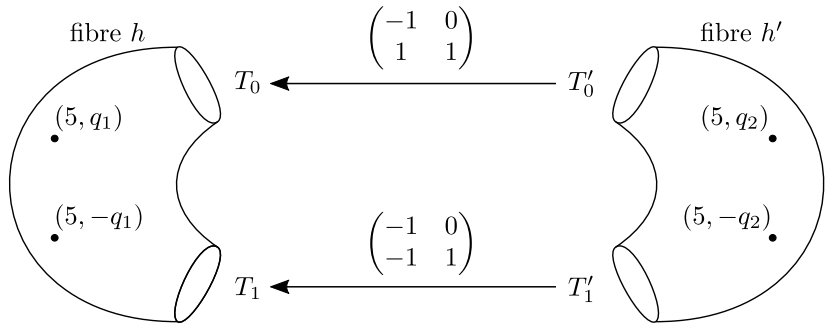}
\caption{Schematic picture of the graph manifold $L_{q_1,q_2}$ considered in Example \ref{ExFibGM}}
\label{FigFibGM}
\end{figure}
\begin{rmk}
The study of fibred manifolds is closely connected with the study of mapping class groups. This will be discussed in Section \ref{secMCG}.
\end{rmk}

%% file: commensurability.tex
\section{Commensurability of graph manifolds}\label{secComm}
The following proposition answers a question asked of me by Michel Boileau at a conference in Marseille. All other known pairs of 3-manifold groups with the same profinite completions are commensurable---both Seifert fibred examples and examples with Sol geometry. The graph manifolds given by Theorem \ref{GMrigid} also fit in this pattern.  
\begin{prop}\label{commens}
Every closed orientable graph manifold has a finite-sheeted cover with profinitely rigid fundamental group. Hence if two graph manifold groups have isomorphic profinite completions, then they are commensurable.
\end{prop}
\begin{proof}
The second statement follows easily from the first. For the first, we will consider the following class of graph manifolds. We will say that a graph manifold $M$ with JSJ decomposition $(X, M_\bullet)$ is {\em right-angled} if the following conditions hold:
\begin{enumerate}[({RA}1)] 
\item the total slope at every vertex space of $M$ is zero,
\item for every edge space the intersection number of the fibres of the adjacent vertex spaces is 1, and
\item every vertex space is of the form $S\times\sph{1}$ for $S$ an orientable surface with genus at least 2 (that is, if all boundary components are filled in with discs, the result has genus at least 2).
\end{enumerate} 
Conditions (RA1) and (RA2) together imply that, after performing suitable Dehn twists, the gluing maps on every torus are simply $\begin{pmatrix} 0 & 1\\1&0\end{pmatrix}$. That is, the fibres of adjacent pieces `are at right angles to each other'. A right-angled graph manifold is thus determined completely by its underlying graph and the first Betti numbers of the vertex groups. This graph, and the first Betti numbers of the vertex spaces, are profinite invariants by Theorem \ref{decompfixed}. Therefore to show that right-angled graph manifolds are indeed profinitely rigid it only remains to show that the property of being right-angled is a profinite invariant.

Theorem \ref{GMrigid}.2(b) immediately shows that having all total slopes zero is a profinite invariant. Theorem \ref{GMrigid}.2(a) shows that property (RA2) is preserved by profinite completions. Finally (RA3) is a profinite invariant by Theorems \ref{Wilkeswbdy} and \ref{decompfixed}. Hence right-angled graph manifolds are indeed profinitely rigid. 

Now consider a closed orientable graph manifold $M$. If the total slope of some vertex is non-zero then we have rigidity by Theorem \ref{GMrigid} so we will ignore this case for the rest of the proof. Also notice that the vanishing of total slopes is essentially equivalent to the vanishing of Euler numbers of the Seifert fibre spaces obtained by filling each boundary torus of a Seifert fibred piece by a solid torus with meridian running along the fibre of the adjacent piece; so this condition is preserved by taking finite-sheeted covers of $M$. We will show that every graph manifold with zero total slope at each vertex has a finite-sheeted cover which is right-angled.  

First pass to a suitable index 1 or 2 cover to eliminate all minor pieces. Every base orbifold of a major Seifert fibre space has a finite cover which is a surface $S$ of genus at least two, hence every Seifert fibred piece has a finite cover of the form $S\times\sph{1}$. The JSJ decomposition is efficient in the profinite topology (Theorem A of \cite{WZ10}) so some finite-sheeted cover of the graph manifold induces such a cover on every Seifert-fibred piece (it may be necessary to pass to a deeper cover on each vertex space than the one specified---but this is still of the required form $S\times\sph{1}$). 

We must now aim to satisfy the condition on intersection numbers. Consider some edge group $T_e=\Z^2$. The fibres of the adjacent vertex groups are primitive elements generating an index $\gamma$ subgroup of $T_e$, so there is a quotient map $T_e\twoheadrightarrow \Z/\gamma\Z$ sending both fibres to zero. The fibres therefore lift to the corresponding degree $\gamma$ cover of the torus, and the intersection number of any choice of such lifts is 1. So if we can find a finite-sheeted cover of the graph manifold inducing this precise cover of each boundary torus, we are done.

For each vertex space $M_v$ we may, as previously noted, fill in each boundary torus $T_e$ by gluing in a solid torus whose meridian is the fibre of the adjacent piece. This gives a Seifert fibre space whose base orbifold is the base orbifold of $M_v$ with each boundary component collapsed and replaced by a cone point of order $\gamma(e)$, the intersection number of the two fibres. The fundamental group of this orbifold is residually finite, so there is a finite quotient into which all the isotropy groups of cone points inject. That is, we have a quotient $G_v\twoheadrightarrow Q_v$ such that the map on each boundary subgroup $T_e$ is precisely the map $T_e\twoheadrightarrow \Z/\gamma(e)\Z$ discussed above.

We may now piece all these quotient groups together into a quotient graph of groups 
\[(X,G_\bullet)\twoheadrightarrow (X, Q_\bullet) \]
where each $Q_v$ is as above and each edge group $Q_e$ is a copy of $\Z/\gamma(e)\Z$. This is a graph of finite groups, hence is residually finite (see for example \cite{Serre}, Section II.2.6, Proposition 11). So there is some finite quotient into which all the finite groups $Q_x$ inject. The corresponding finite-sheeted cover of the graph manifold is the required rigid cover.
\end{proof}

%% file: knots.tex
\section{Behaviour of knot complements}\label{secKnots}
In this section we will prove the following theorem. A {\em graph knot} is a knot in $\sph{3}$ whose exterior is a graph manifold.
\begin{theorem}\label{thmKnots}
Let $M_K$ be the exterior of a graph knot $K$. Let $N$ be another compact orientable 3-manifold and assume that $\widehat{\pi_1 M_K}\iso \widehat{\pi_1 N}$. Then $\pi_1 M_K\iso \pi_1 N$. In particular if $K$ is prime and $N$ is also a knot exterior then $N$ is homeomorphic to $M_K$.
\end{theorem}

While the theorems of \cite{Wilkes16} were stated for closed manifolds, for most of the paper the assumption that the boundary was empty was not relevant. The arguments prior to Section 10 of \cite{Wilkes16} were based on the existence and properties of a JSJ decomposition along tori. This structure, and its properties, holds just as well for manifolds with toroidal boundary. In particular for graph knot exteriors we find that an isomorphism of profinite fundamental groups still induces a isomorphism of JSJ graphs and isomorphisms of the profinite fundamental groups of each vertex space of the JSJ decomposition. It is not necessarily true that any such isomorphism preserves the fundamental group of the boundary component (even up to conjugacy). Indeed this is not even true for the discrete fundamental group. However we shall see that the only ambiguities in determining the manifold from its profinite fundamental group come from ambiguities in the discrete fundamental group. 

Before proving the theorem we will need to discuss in more detail the Seifert fibre spaces arising in the JSJ decomposition of a graph knot exterior. Statements made in this discussion will be used without explicit reference in the following proofs.

Since a loop in the JSJ graph would give rise to a non-trivial element of $H^1(M_K; \Z)$ vanishing on the boundary component of $M_K$, which is impossible by standard properties of knots, the JSJ graph is a rooted tree with root given by the single boundary component. It follows from Section 7 and Corollary 9.3 of \cite{EN85} that the only possible vertex spaces are those described in the following list. A paraphrase of these results would be that all graph knots are built up from torus knots by the operations of cabling and connected-summation. See also Proposition 3.2 of \cite{Budney05}. Note that there are additional possibilities when considering exteriors of graph {\em links}. We do not consider this issue here.

\subsubsection*{Torus knot exteriors}
The exterior $E_{p,q}$ of a $(p,q)$-torus knot, for $|p|, |q|\geq 2$. This is a Seifert fibre space with two exceptional fibres of orders $p$ and $q$. Since $p$ and $q$ are coprime there exist $\bar p$ and $\bar q$ such that $\bar p p + \bar q q=1$. Then a presentation for the fundamental group is
\[\big< a,b,h \,\big|\, h \text{ central,}\, a^ph^{\bar q}, b^qh^{\bar p}\big> \]
where $h$ is the homotopy class of a regular fibre and $ab$ is a meridian curve of the knot. We remark that replacing $\bar q$ and $\bar p$ by any integers coprime to $p$ and $q$ respectively does not change this group up to isomorphism.

The profinite completion of this group has centre $\widehat\Z$ generated by $h$, and the quotient by this centre is $\Z/p\amalg \Z/q$, where $\amalg$ denotes the free profinite product. By the techniques of \cite{Wilkes15} any Seifert fibre space group with the same profinite completion has the same base orbifold (a disc with two cone points of orders $p$ and $q$), and will therefore have the same discrete fundamental group by the above remark. It will also have only one boundary component. See also \cite[Section 4]{GZ11}.

Theorem \ref{Wilkeswbdy} implies that the Seifert fibre spaces with the same profinite fundamental group as $E_{p,q}$ {\em preserving peripheral structures} are precisely those with the Seifert invariants ${\bar q}$ and $\bar p$ replaced by $k{\bar q}$ and $k\bar p$ for any $k$ coprime to $pq$. Note that the requirement $\bar p p + \bar q q=1$ additionally shows that, for $k$ not congruent to 1 modulo $pq$, this Seifert fibre space is not a knot exterior.

If a torus knot exterior arises as a JSJ piece of a graph knot exterior then this piece must of course be a leaf of the JSJ tree. As the other possibilities in the list will show, the converse also holds: any leaf (except the root) is a torus knot exterior.

\subsubsection*{Products}
Pieces of the form $S\times\sph{1}$, where $S$ is a sphere with at least $k+1\geq 3$ open discs removed. The only Seifert fibre spaces with the same profinite fundamental group as $S\times\sph{1}$ also have the same discrete fundamental group---see Theorem 5.5 of \cite{Wilkes15}---and are therefore also of the form $F\times \sph{1}$ for some surface $F$ with $\pi_1 F\iso \pi_1 S$. Note that if $F$ is orientable then either $F=S$ or $F$ has at most two fewer boundary components than $S$. 

The presence of such a piece in the JSJ decomposition of a graph knot exterior represents the procedure of taking the connected sum of several graph knots. Note that under this operation, the {\em meridian} of each summand becomes a {\em fibre} of the product piece. Such a piece is either of valence $k+1$ in the JSJ graph, or $k$ if it happens to be the root piece.

\subsubsection*{Cable spaces}
A cable space $C_{s,t}$ of type $(s,t)$, for $|s|\geq 2$, consists of the space formed from a fibred solid torus $T$ with Seifert invariants $(s,t)$ by removing a neighbourhood of a regular fibred removed. Equivalently this is the orientable Seifert fibre space with base orbifold an annulus with a single cone point of Seifert invariants $(s,t)$ where $s$ is coprime to $t$. A presentation for the fundamental group is
\[\pi_1 C_{s,t}=\big< c,e,j\,\big|\, j \text{ central, } c^sj^t\big>\]
where the regular fibre is $j$ and the boundary components are given by the conjugacy classes of the subgroups $\gp{j,e}$ and $\gp{j,(ce)^{-1}}$. Note also that, if the boundary component of $T$ is represented in $\pi_1C_{s,t}$ by $\gp{j,e}$ then the boundary of a meridian disc of $T$ is given by $m=j^{t}e^{-s}$.

Notice that these groups $\pi_1 C_{s,t}$ as $t$ varies are abstractly isomorphic to each other, and that such isomorphisms can be chosen to fix one (but not both) boundary components. Theorem \ref{Wilkeswbdy} shows that the profinite fundamental groups of $C_{s,t}$ as $t$ varies are isomorphic while preserving all of the given peripheral structure. 

Furthermore (the proof of) Theorem 5.5 of \cite{Wilkes15} shows that the only orientable Seifert fibre spaces with the same profinite fundamental group (ignoring the peripheral structure) as a cable space have orientable base orbifolds with fundamental group $\Z\ast \Z/s$. There is only one such orbifold, so the Seifert fibre space in question is again a cable space (with the same invariant $s$).

The presence of a cable space piece in the JSJ decomposition of a graph knot exterior represents the cabling operation on knots. The corresponding vertex of the JSJ graph has valence two, or one if it happens to be the root. Note that in this case the meridian curve of the initial knot becomes the element $m$ given above (i.e.\ the meridian curve of $T$). A longitude of the initial knot will therefore be given by $l=j^{\bar s}e^{\bar t}$ for some integers $\bar s, \bar t$ such that $s\bar s+t\bar t=1$. The ambiguity in the choice of these integers reflects the ambiguity in the choice of longitude. 

\begin{lem}\label{nonzeroslope}
Let $K$ be a graph knot with exterior $M_K$. Let $(X,M_\bullet)$ be the JSJ decomposition of $M_K$, where $X$ is viewed as a rooted tree. Let $x$ be a leaf of $X$ that is not the root. Then the total slope at $x$ is non-zero.
\end{lem}
\begin{proof}
Since $x$ is a leaf, the corresponding manifold $M_x$ is the exterior of a $(p,q)$ torus knot for some coprime integers $p$ and $q$. As above choose a presentation
\[\pi_1 M_x=\big< a,b,h \,\big|\, h \text{ central,}\, a^ph^{\bar q}, b^qh^{\bar p}\big> \]
where $\bar p p + \bar q q=1$ and the meridian of the torus knot is $ab$. Let $y$ be the unique vertex of $X$ adjacent to $x$. There are two cases to consider: whether $M_y$ is a cable space or a product. 

Suppose first that $M_y$ is a product. Then by the discussions above the meridian $ab$ of  the torus knot is isotopic to the fibre $j$ of the product. Thus the gluing matrix along the edge joining $x$ and $y$, oriented from $x$ to $y$, has the form 
\[ \begin{pmatrix}
\alpha & \beta\\ \gamma & \delta
\end{pmatrix} 
= \begin{pmatrix}
\alpha & 1 \\ 1 & 0
\end{pmatrix}\]
Hence the total slope at $x$ is 
\[ \tau(x) = 0-\frac{\bar q}{p} - \frac{\bar p}{q} = -\frac{1}{pq} \neq 0\]

Next we consider the case when $M_y$ is an $(s,t)$-cable space whose fundamental group has presentation
\[\pi_1 C_{s,t}=\big< c,e,j\,\big|\, j \text{ central, } c^sj^t\big>\]
where the regular fibre is $j$ and the boundary components are given by the conjugacy classes of the subgroups $\gp{j,e}$ and $\gp{j,(ce)^{-1}}$. Without loss of generality let the boundary component glued to $M_x$ be $\gp{j,e}$. As discussed above the meridian $ab$ of the torus knot is given by $ab=j^{t}e^{-s}$ and the fibre $h$---which is a longitude of the torus knot---is given by $h=j^{\bar s}e^{\bar t}$ for some integers $\bar s, \bar t$ such that $s\bar s+t\bar t=1$. 

 Thus the gluing matrix along the edge joining $x$ and $y$, again oriented from $x$ to $y$, has the form 
\[ \begin{pmatrix}
\alpha & \beta\\ \gamma & \delta
\end{pmatrix} 
= \begin{pmatrix}
\bar s & t \\ \bar t & -s
\end{pmatrix}\]
And the vanishing of the total slope would imply
\[0 = pq{\bar t}\tau(x)=pq(-s) - \bar t( \bar q q + \bar p p ) = -pqs  - \bar t\]
and hence $s$ would divide $\bar t$, so $s=1$ giving a contradiction. So the total slope is non-zero as required. 
\end{proof}

\begin{proof}[Proof of Theorem \ref{thmKnots}]
Let $(X,M_\bullet)$ and $(Y,N_\bullet)$ be the JSJ decompositions of $M_K$ and $N$ respectively. Let $G=\pi_1 M_K$, $G_\bullet=\pi_1 M_\bullet$, $H=\pi_1 N$ and $H_\bullet =\pi_1 N_\bullet$. Let $\Phi\colon \widehat{G}\to\widehat{H}$ be an isomorphism. By Theorem \ref{decompfixed} and Theorem 7.1 of \cite{Wilkes16} (whose proofs do not rely on the manifolds being closed) we find that $N$ is a graph manifold and, possibly after post-composing $\Phi$ with some automorphism of $\widehat{H}$, there is a graph isomorphism $\phi\colon X \to Y$ such that $\Phi$ restricts to an isomorphism $\widehat{G}_x\iso \widehat{H}_{\phi(x)}$ for each $x\in X$. By the discussions above, this implies that $G_x$ and $H_{\phi(x)}$ are abstractly isomorphic.

Now $X$ is a rooted tree with root $r$ distinguished by the single boundary component of $M_K$. For standard cohomological reasons (for example, Corollary 4.2 of \cite{BF15}) the manifold $N$ cannot be closed. We claim that $N$ has exactly one boundary component, located in the piece $N_{\phi(r)}$. We do not claim---indeed it may not be true even for isomorphisms of discrete fundamental groups---that this boundary component $\bdy N$ satisfies
\[\widehat{\bdy N}=\Phi(\widehat{\pi_1(\bdy M_K)}), \]
even up to conjugacy. However its position in the JSJ decomposition is fixed because of the following argument. Here a `free boundary component' of a JSJ piece of a 3-manifold will mean a boundary torus not glued to any JSJ piece---that is, those boundary components which survive in the boundary of the ambient manifold.

Every Seifert fibre space with fundamental group isomorphic to that of a torus knot exterior has exactly one boundary component---and every such piece of $N$ has valence 1 in $Y$ because of the isomorphism $\Phi$, hence has no free boundary components. Similarly every Seifert fibre space with fundamental group isomorphic to that of a cable space has two boundary components, and all such pieces $N_y$ have valence 2 in $Y$ unless $y=\phi(r)$, when there is one free boundary component $\bdy N$. Comparing $X$ and $Y$, every piece $N_y$ with fundamental group isomorphic to $F_k\times\Z$ has valence $k+1$ in $Y$ unless $y=\phi(r)$, when it has valence $r$. Now any orientable Seifert fibre space with fundamental group $F_k\times\Z$ has either $k+1$ boundary components or has strictly fewer than $k$---hence all these pieces have no free boundary components except if $y=\phi(r)$ when it has exactly one. So we see that $N$ may have at most one boundary component, and it is located in the piece $N_{\phi(r)}$.

Now consider the isomorphisms $\Phi_x\colon \widehat{G}_x\to \widehat{H}_{\phi(x)}$. For $x\neq r$, the fact that $\Phi$ preserves the JSJ decomposition implies that the peripheral structure is preserved by $\Phi_x$. So by Theorem \ref{Wilkeswbdy} the isomorphism $\Phi_x$ determines constants $\lambda_x,\mu_x\in\widehat\Z {}^{\!\times}$ such that $\lambda_x$ gives the map from the fibre subgroup of $M_x$ to the fibre subgroup of $N_{\phi(x)}$ and the map on base orbifolds is an exotic isomorphism of type $\mu_x$.  Analysing the gluing maps in $M$ and $N$ as in the proof of Theorem \ref{GMrigid} we find that if $x'$ is adjacent to $x$ (and neither is the root) we have $\lambda_x=\mu_{x'}$ and $\mu_x = \lambda_{x'}$, perhaps up to choosing appropriate orientations on fibre subgroups to eliminate minus signs.  Furthermore the total slopes of $M_x$ and $N_{\phi(x)}$ are related by multiplication by $\lambda_x/\mu_x$. (Strictly speaking we have not defined multiplication of an element of $\Q$ by an element of $\widehat\Z$; we really mean that after clearing denominators by an integer $n\in\Z$ we have $n\tau(\phi(x))= \lambda_x\mu_x^{-1}\tau(x)$. Alternatively one can consider the ring obtained from $\widehat{\Z}$ by inverting all the elements of \Z.)

However the total slope $\tau(x)$ of any piece $M_x$ which is a torus knot exterior is non-zero by Lemma \ref{nonzeroslope}. So the total slope of $N_{\phi(x)}$ is a rational number equal to $\lambda_x/\mu_x$ times a non-zero rational number. This implies that $\lambda_x = \mu_x$ (up to changes in orientations) using Lemma 2.2 of \cite{Wilkes16}. Every connected component of $X\smallsetminus\{r\}$ contains a leaf of $X$, hence has a vertex space which is a torus knot exterior. It follows that $\lambda_\bullet$ and $\mu_\bullet$ are constant on connected components of $X\smallsetminus\{r\}$.

Now on the root pieces of $M$ and $N$, the fibre subgroup is the unique maximal central subgroup of the relevant profinite fundamental group by Theorem 5.5 of \cite{Wilkes15} and so is preserved by $\Phi_r$. Let the map on fibre subgroups be multiplication by $\lambda_r$ (as usual, identifying this fibre subgroup with $\widehat\Z$ via a generator in the discrete fundamental group). Again, examining the gluing maps along the tori gluing $r$ to other vertices and using the same analysis as in Theorem \ref{GMrigid} we find that if $x$ is adjacent to $r$ then $\pm\mu_x=\lambda_r$. Therefore we find that $\lambda_\bullet$ and $\mu_\bullet$ are constants over all of $X$ (up to reversing all the orientations on components of $X\smallsetminus \{r\}$ to fix the signs). 

By Theorem \ref{Wilkeswbdy}, the fact that $\lambda_x/\mu_x=1$ (together with the fact that away from the root all boundary tori of JSJ pieces are edge groups in the JSJ decomposition, hence are preserved by $\Phi$) implies that $M_{x}$ and $N_{\phi(x)}$ are homeomorphic for $x\neq r$. 

Now, the fundamental group of the graph manifold $N$ is determined up to isomorphism by the following data: the JSJ graph; the isomorphism types of group pairs given by each vertex group and its adjacent edge groups; the intersection numbers $\gamma(e)$ of adjacent fibres, for each edge $e$ of the JSJ graph; the invariants $\delta(e)$ modulo $\gamma(e)$; and the total slope of each vertex space that is not the root. These allow us to reconstruct the total group uniquely---as explained in the preliminaries, the indeterminacy of $\delta(e)$ modulo $\gamma(e)$ may be resolved using Dehn twists given that the total slope is fixed. For the root one may use the free boundary component for Dehn twists, hence no sort of total slope condition is required.

Now the analysis in Theorem 10.1 of \cite{Wilkes16} shows that the fact that $\lambda_x/\mu_x$ is always equal to 1 fixes all of these data for $N$ to be equal to those for $M$. Hence the manifolds have the same fundamental group.

In the case when $K$ is prime (that is, the root piece is a cable space) and $N$ is also a knot exterior, then the fundamental group determines the homeomorphism type and final statement of the theorem also follows. See Section 7 and Corollary 9.3 of \cite{EN85} for the proof of the for graph knots; more generally is \cite[Corollary 2.1]{GL89}.
\end{proof}
\begin{rmk}
We comment that the properties deriving from the fact that $M_K$ was a knot exterior were crucial to the rigidity in this theorem.  In particular the fact that every leaf of the JSJ tree had non-zero total slope provides strong rigidity to all complements of the root piece, leaving little flexibility in what remained. If one had some complement of the root piece which was not profinitely rigid (relative it the boundary component joining it to the root) one can easily extend this to further non-rigid examples. The free boundary component in the root means that results such as Theorem \ref{Wilkeswbdy} which forces all the boundary components of the root to behave in roughly the same way simply do not apply. Possibly one could impose extra constraints on the boundary to extend the analysis of Theorem \ref{GMrigid}. However a large part of the interest of Theorem \ref{thmKnots} is that no such boundary condition is needed.
\end{rmk}

%% file: mcg.tex
\section{Relation to mapping class groups}\label{secMCG}
In this section, we will only discuss closed orientable manifolds, to avoid worries on the boundary. We will view a fibred 3-manifold $(M, \zeta)$ as a 3-manifold $M$ equipped with a choice of homomorphism $\zeta\colon\pi_1 M \twoheadrightarrow \Z$ with finitely generated kernel $\pi_1 S$, where $S$ is a closed orientable surface. By Stallings' theorem on fibred 3-manifolds \cite[Theorem 2]{stallings62}, this is equivalent to the topological definition. For such a fibred manifold, $\pi_1 M$ has many expressions as a semidirect product $\pi_1 S\rtimes_\phi \Z$ each given by a section of $\zeta$. The different maps $\phi$ differ by composition with an inner automorphism of $\pi_1 S$, hence give a well-defined element of $\Out(\pi_1 S)$. If we have two automorphisms $\phi_1$ and $\phi_2$ of $\pi_1 S$, these are conjugate in ${\rm Aut}(\pi_1 S)$ by some automorphism $\psi$ if and only if there is an isomorphism of semidirect products
\[ (\psi, \id)\colon \pi_1 S \rtimes_{\phi_1} \Z \to \pi_1 S \rtimes_{\phi_2} \Z  \]    
Allowing for a change in section, we find that $\phi_1$ and $\phi_2$ are conjugate in $\Out(\pi_1 S)$ if and only if there is a commuting diagram
\[\begin{tikzcd}
\pi_1 M_1\ar{d}{\Psi}\ar[two heads]{r}{\zeta_1} & \Z \ar{d}{\id}\\
\pi_1 M_2\ar[two heads]{r}{\zeta_2} & \Z
\end{tikzcd}\]
where $(M_1,\zeta_1)$ and $(M_2, \zeta_2)$ are the fibred manifolds corresponding to $\phi_1$ and $\phi_2$.

All the above equivalences still hold when one replaces all manifold groups with their profinite completions and ${\rm Aut}(\pi_1 S)$ with ${\rm Aut}(\widehat{\pi_1 S})$. There is a canonical injection 
\[{\rm Aut}(\pi_1 S) \hookrightarrow {\rm Aut}(\widehat{\pi_1 S}) \]
and we will abuse notation by identifying an automorphism of $\phi_1 S$ with the induced automorphism of the profinite completion. As was noted by Boileau and Friedl \cite[Corollary 3.6]{BF15s}, Theorem 5.2 of \cite{Wilkes15} implies that the canonical map
\[\Out(\pi_1 S) \to \Out(\widehat{\pi_1 S}) \]
is also an injection.

Thus related to the question of whether two fibred manifolds can have isomorphic profinite fundamental groups we have a question concerning automorphisms of surface groups.
\begin{question}\label{QnPrMCG}
Do there exist automorphisms $\phi_1$, $\phi_2$ of a surface group $\pi_1 S$ which are conjugate in $\Out(\widehat{\pi_1 S})$ but not in $\Out(\pi_1 S)$?
\end{question}

When $S$ is a once-punctured torus, this question has been proven to have a negative answer by Bridson, Reid and Wilton \cite[Theorem A]{BRW16}. 

While a positive answer to Question \ref{QnPrMCG} would give examples of fibred manifold groups with isomorphic profinite completions, the converse does not always hold; one could conceivably have profinite isomorphisms of the manifold groups which do not in any sense preserve the fibrations. For instance Example \ref{ExFibGM} above does not give a positive solution to Question \ref{QnPrMCG}; as we will see below, no other graph manifold does either. Let us make this precise. 
\begin{defn}
Let $(M, \zeta_M)$ and $(N, \zeta_N)$ be fibred graph manifolds. Suppose \[\Psi\colon \widehat{\pi_1 M}\to \widehat{\pi_1 N}\] is an isomorphism. We say that $\Psi$ {\em weakly preserves the fibration} if there is a commuting diagram 
\[\begin{tikzcd}
\widehat{\pi_1 M}\ar{d}{\Psi}\ar[two heads]{r}{\hat\zeta_M} & \widehat\Z \ar{d}{\cdot\kappa}\\
\widehat{\pi_1 N)}\ar[two heads]{r}{\hat\zeta_N} & \widehat\Z
\end{tikzcd}\]
for some $\kappa\in\widehat\Z {}^{\!\times}$. We say that $\Psi$ {\em strongly preserves the fibration} if there exists such a diagram with, additionally, $\kappa= +1$.    
\end{defn}
Relating this definition to the discussion above, we see that strong fibre preservation yields conjugacy in $\Out(\widehat{\pi_1 S})$. Weak fibre preservation says that $\phi_1$ is conjugate to $\phi_2^\kappa$ in $\Out(\widehat{\pi_1 S})$.

Finite order automorphisms give rise to Seifert fibre spaces of geometry $\HR$. Hempel's original paper giving examples of \SFS{}s which are not profinitely rigid do not give an explicit isomorphism of profinite groups, so do not say anything about mapping classes. The isomorphisms constructed in \cite{Wilkes15} only preserve the fibre weakly. For a suitable choices of sections, the mapping classes involved are $\phi^k$ and $\phi$ for some $k$ which is coprime to $n$ equal to the order of $\phi$. The weak fibre preservation then says that $\phi^k$ is conjugate to $\phi^\kappa$ in $\Out(\pi_1 S)$, where $\kappa$ is the factor by which we stretch the homotopy class $h$ of a regular (Seifert) fibre. This is not exactly surprising, since in fact these automorphisms are equal. Indeed this gives yet another way to see that the corresponding \SFS{} groups must have isomorphic profinite completions. 

We can however say more. The mapping classes $\phi^k$ and $\phi$ are genuinely conjugate in $\Out(\widehat{\pi_1 S})$---that is, there exists an isomorphism which {\em strongly} preserves the fibration. The isomorphisms in \cite{Wilkes15} do not do this, but armed with the exotic automorphisms of surface groups from Proposition \ref{orbexauto} we can build new isomorphisms. The case with non-empty boundary was covered by Theorem 10.7 of \cite{Wilkes16} as the focus was then on constructing graph manifolds. We now deal with the case of closed \SFS{}s. 
\begin{theorem}
Let $S$ be a closed hyperbolic surface and let $\varphi$ be a periodic self-homeomorphism of $S$. Let $k$ be coprime to the order of $\varphi$. Let $(M,\zeta)$ be the surface bundle with fibre $S$ and monodromy $\varphi$ and let $(M',\zeta')$ be the surface bundle with fibre $S$ and monodromy $\varphi^k$. Then there is an isomorphism \[ \Psi\colon \widehat{\pi_1 M}\to \widehat{\pi_1 M'}\]
which strongly preserves the fibration.
\end{theorem}
\begin{proof}
We may choose presentations for these two \SFS{}s in the standard form. Note that since $M$ is fibred over the circle, the base orbifold is orientable and the geometry is \HR. So let 
\begin{multline*}
\pi_1 M = \big< a_1,\ldots, a_r, u_1, v_1, \ldots, u_g, v_g, h\,\big|\\ h\in Z(\pi_1 M),\, a_i^{p_i} h^{q_i}, \,a_1\ldots a_r[u_1, v_1]\ldots [u_g, v_g] = h^b \big>  
\end{multline*}
where $b= -\sum q_i / p_i$ since the Euler number vanishes. The map $\zeta$ is given by
\[h\mapsto \prod_j p_j, \quad a_i\mapsto -q_i\prod_{j\neq i} p_j \] 
and by sending $u_i, v_i$ to zero. Similarly define 
\begin{multline*}
\pi_1 M = \big< a_1,\ldots, a_r, u_1, v_1, \ldots, u_g, v_g, h\,\big|\\ h\in Z(\pi_1 M),\, a_i^{p_i} h^{q'_i}, \,a_1\ldots a_r[u_1, v_1]\ldots [u_g, v_g] = h^{b'} \big>  
\end{multline*}
where $b'= -\sum q'_i / p_i$. Again the map $\zeta'$ is given by
\[h\mapsto \prod_j p_j, \quad a_i\mapsto -q'_i\prod_{j\neq i} p_j \] 
Note that by construction  we have $q_i \equiv \kappa q'_i$ modulo $p_i$ for every $i$, for $\kappa\in\widehat\Z {}^{\!\times}$ congruent to $k$ modulo the order of $\varphi$. See Proposition 5.1 of \cite{hempel14} for the translation of data from surface bundles to Seifert invariants. Let $\rho_i\in\widehat\Z$ be such that $q_i = \kappa q'_i + \rho_i p_i$. Now by Proposition \ref{orbexauto} there is an automorphism $\psi$ of the free group on the generators $\{a_i,u_i,v_i\}$ which sends each $a_i$ to a conjugate of $a_i^\kappa$, sends every other generator to a conjugate of a power of itself, and sends 
\[a_1\ldots a_r[u_1, v_1]\ldots [u_g, v_g]\mapsto \left(a_1\ldots a_r[u_1, v_1]\ldots [u_g, v_g]\right)^\kappa\]
Now define $\Psi\colon \widehat{\pi_1 M}\to\widehat{\pi_1 M}$ by
\[h\mapsto h, \quad a_i\mapsto \psi(a_i)h^{-\rho_i}, \quad u_i\mapsto \psi(u_i), \quad v_i\mapsto \psi(v_i) \]
The reader may readily check that this map $\Psi$ is a well-defined isomorphism of profinite groups and that $\hat\zeta = \hat \zeta' \Psi$ as required.
\end{proof}
\begin{theorem}\label{thmMCG}
If $M$ and $N$ are non-homeomorphic closed fibred graph manifolds, and $\Psi\colon\widehat{\pi_1 M}\to\widehat{\pi_1 N}$ is any isomorphism, then $\Psi$ does not preserve any fibrations of $M$ and $N$, even weakly. Hence if $\phi_1$ and $\phi_2$ are piecewise periodic automorphisms of a closed surface group $\pi_1 S$ which are not conjugate in $\Out(\pi_1 S)$, then $\phi_1$ is not conjugate to any power of $\phi_2$ in $\Out(\widehat{\pi_1 S})$.
\end{theorem}
\begin{rmk}
Here the `piecewise periodic' means that the corresponding fibred manifold is a graph manifold. The author is not aware of a standard term for this.
\end{rmk}
\begin{proof}
For suppose $(M, \zeta_M)$ and $(N, \zeta_N)$ are fibred graph manifolds and $\Psi$ is an isomorphism of the profinite completions of their fundamental groups weakly preserving the fibration. Suppose $M$ and $N$ are not homeomorphic. Let the JSJ decompositions be $(X,M_\bullet)$ and $(Y,N_\bullet)$ and denote the regular fibre of a vertex group $\pi_1 M_x$ or $\pi_1 N_x$ by $h_x$. Note that any fibre surface in a graph manifold must cut the Seifert fibres of all vertex spaces transversely, so every Seifert fibre survives as a non-trivial element of $\Z$ under the map $\zeta_M$ (or $\zeta_N$). Also, for an orientable fibred graph manifold all the base orbifolds of major pieces must be orientable. We may therefore apply the analysis in the proof of Theorem 10.1 of \cite{Wilkes16} to conclude that there is a graph isomorphism $\psi\colon X\to Y$ and numbers $\lambda, \mu\in\widehat\Z {}^{\!\times}$ such that for adjacent vertices $x$ and $y$ are of $X$, we have (where $\sim$ denotes conjugacy in $\widehat{\pi_1 N}$):
\[ \Psi(h_x) \sim h_{\psi(x)}^\lambda, \quad \Psi(h_y) \sim h_{\psi(y)}^\mu \]
(or vice versa). Since $M$ and $N$ are not homeomorphic the ratio $\lambda/\mu$ is not equal to $\pm 1$. Now if $\Psi$ weakly preserves the fibre, then we have equations 
\[ \kappa \zeta_M(h_x) = \zeta_N(\Psi(h_x)) = \zeta_N(h_{\psi(x)}^\lambda) = \lambda\zeta_N(h_{\psi(x)})  \]
so that $\lambda/\kappa$, when multiplied by a non-zero element of \Z, remains in \Z. Thus by standard theory (for example Lemma 2.2 of \cite{Wilkes16}), $\kappa=\pm \lambda$. Applying the same argument to $y$ gives $\kappa=\pm \mu$, so that $\lambda/\mu=\pm 1$, giving a contradiction.
\end{proof}
\begin{clly}
A closed fibred graph manifold of first Betti number one is profinitely rigid.
\end{clly}
\begin{proof}
By Theorem 7.1 of \cite{Wilkes16} and Theorem 1.1 of \cite{Jaikin17} any other manifold with the same profinite fundamental group as the given manifold is also a closed fibred graph manifold. If the first Betti number is 1 then there is a unique map to $\widehat\Z$ so that any isomorphism weakly preserves the fibration. Theorem \ref{thmMCG} now gives the result.
\end{proof}
\begin{rmk}
There is another proof of this corollary, which we will now sketch, which is more closely related to Theorem \ref{GMrigid}. For consider a closed fibred graph manifold $M$ of first Betti number one. Then there is an essentially unique homomorphism $\zeta\colon\pi_1 M \to \Q$ which, as noted above, does not vanish on the regular fibre of any Seifert-fibred piece of $M$. This in turn implies that the JSJ graph of $M$ is a tree and that the base orbifolds of all pieces are spheres with cone points with discs removed. 

Consider piece $M_l$ corresponding to a leaf $l$ in the JSJ graph. We will show that the total slope at $M_l$ is non-zero, so that $M$ is profinitely rigid by Theorem \ref{GMrigid}. Let $e$ be the edge emanating from $l$, let $h$ be the homotopy class of a regular fibre of $M_l$, and let $h'$ be the regular fibre of $d_1(e)$. If $\pi_1 M_l$ has a standard form presentation
\[\big< a_1,\ldots, a_r, h\,\big|\, a_i^{p_i} h^{q_i}, h \text{ central}\, \big>\] 
then, if $e_0 = \left(a_1 \cdots a_r\right)^{-1}$ we have 
\[h' = e_0^{-\gamma(e)} h^{\delta(e)} = \left(a_1 \cdots a_r\right)^{\gamma}h^\delta \]
where $\gamma\neq 0$. Without loss of generality suppose $\zeta(h)=1$. Then $\zeta(a_i) = -q_i/p_i$ and 
\[\zeta(h') = -\gamma \sum \frac{q_i}{p_i} + \delta = \gamma \tau(l)\]
Since $\zeta(h')\neq 0$ we find that the total slope at $l$ is non-zero as claimed.
\end{rmk}

Theorem \ref{thmMCG} shows that Theorem \ref{GMrigid}, while finding examples of non-rigid fibred manifolds, leaves open the possibility that Question \ref{QnPrMCG} could have a negative answer for all infinite order mapping classes.  It also raises the possibility that even if profinite rigidity for hyperbolic 3-manifolds fails, the weaker statement about mapping class groups could still hold.

\begin{rmk}
It is curious to compare the directions of the proofs in this section and in \cite{BRW16}. In the latter paper, strong properties of $\Out(F_2)$ (`congruence omnipotence for elements of infinite order') were used to deduce that independent mapping classes $\phi_1, \phi_2\in\Out(F_2)$ were not conjugate to any power of each other in $\Out(\widehat F_2)$, so that no isomorphism of the profinite completions of once-punctured torus bundles could (in our terminology) weakly preserve the fibration. Meanwhile the assumption that the first Betti number is 1 showed that any such isomorphism must weakly preserve the fibration. In this way \cite{BRW16} obtained a profinite rigidity theorem for once-punctured torus bundles.

In our situation, the direction is quite different; we investigated profinite completions of graph manifolds, and in doing so learnt about conjugacy of certain elements in $\Out(\widehat{\pi_1 S})$. One rather suspects that the results about the mapping class group should come first, but seem to be lacking except in the case described above.
\end{rmk}